\documentclass[a4paper, 11pt]{amsart}
\usepackage[margin=3.5cm]{geometry}
\usepackage[utf8]{inputenc}
\usepackage{amsmath,amssymb,amsthm,amscd,amsfonts,mathrsfs,verbatim, mathtools}
\usepackage{stmaryrd}
\usepackage{enumitem}
\usepackage[all]{xy}
\usepackage{tikz}
\usetikzlibrary{decorations.pathreplacing}

\newcommand{\C}{{\mathbb C}}                   
\newcommand{\NN}{{\mathbb N}}                   
\newcommand{\ZZ}{{\mathbb Z}}                   
\newcommand{\QQ}{{\mathbb Q}}                   
\newcommand{\CC}{{\mathbb C}}                   

\newcommand{\CP}[1]{\mathbb{P}^{#1}}      

\DeclareMathOperator{\Spec}{Spec}
\DeclareMathOperator{\vdim}{vdim}

\newtheorem{theorem}{Theorem}[section]

\newtheorem{proposition}[theorem]{Proposition}
\newtheorem{corollary}[theorem]{Corollary}
\newtheorem{lemma}[theorem]{Lemma}

\newtheorem{defn}[theorem]{Definition}

\theoremstyle{definition}
\newtheorem{definition}[theorem]{Definition}

\theoremstyle{remark}
\newtheorem{remark}[theorem]{Remark}

\newtheorem{example}[theorem]{Example}

\title{The quantum D-module of product varieties}

\makeatletter
\@namedef{subjclassname@2020}{%
  \textup{2020} Mathematics Subject Classification}
\makeatother

\author{\'Ad\'am Gyenge}
\address{Department of Algebra and Geometry, Institute of Mathematics,  
Budapest University of Technology and Economics, 
M\H{u}egyetem rakpart 3, 1111, 
Budapest, Hungary}
\email{Gyenge.Adam@ttk.bme.hu}


\begin{document}

\begin{abstract}
We study the quantum connection of product varieties in the framework of quantum cohomology. Our first main result shows that, near the origin of the Novikov variables, the quantum spectrum of \(X \times Y\) converges to the set of pairwise sums of the spectra of \(X\) and \(Y\). This arises from the leading contribution of the connection matrices \(K_X \otimes \mathrm{id}\) and \(\mathrm{id} \otimes K_Y\), while mixed curve classes contribute only at higher order. Our second main result establishes a formal isomorphism of quantum \(D\)-modules $
\mathrm{QDM}(X \times Y)^{\mathrm{la}} \cong \mathrm{QDM}(X)^{\mathrm{la}} \otimes \mathrm{QDM}(Y)^{\mathrm{la}}$,
compatible with the quantum connection. As applications, we show that atoms, birational invariants arising from quantum cohomology, factor multiplicatively for product varieties, and we deduce the existence of a motivic measure associated with atoms.
\end{abstract}

\maketitle

\tableofcontents 

\section{Introduction}

Quantum cohomology has emerged as an important tool in modern algebraic geometry, symplectic geometry, and mathematical physics. It enriches classical cohomology by incorporating enumerative information about rational curves in a smooth projective variety \(X\). More precisely, the quantum cohomology ring \(QH^*(X)\) is a deformation of the classical cohomology ring, whose structure constants encode genus-zero Gromov--Witten invariants. These invariants count, in a virtual sense, the number of rational curves in \(X\) passing through prescribed cycles and subject to homology constraints. Quantum cohomology plays a crucial role in mirror symmetry and the study of Frobenius manifolds.

A fundamental object associated with quantum cohomology is the \emph{quantum connection}, also known as the Dubrovin connection. This is a flat connection on the trivial vector bundle over the space of Novikov variables, whose connection matrices encode quantum multiplication by the Euler vector field and by elements of the cohomology of \(X\). The spectrum of the quantum connection (shortly, the \emph{quantum spectrum}) is particularly significant: it governs the asymptotic behavior of flat sections, encodes semisimplicity of the quantum cohomology, and appears in various enumerative and integrable contexts. We apply asymptotic analysis to understand this spectrum by considering degenerating limits of the Novikov variables.

The first main result of this paper concerns the quantum connection of a product variety \(X \times Y\). We show that, in the neighborhood of the origin of the Novikov variables, the eigenvalues of the quantum connection for \(X \times Y\) are given by sums of the eigenvalues of the quantum connections for \(X\) and for \(Y\). 

\begin{theorem}[{Theorem~\ref{thm:main1_eiglim}}]
Let 
\[
\lambda_0 (q_1,t_1), \dots, \lambda_{p} (q_1,t_1)
\] 
denote the quantum spectrum of $X$, and 
\[
\mu_0 (q_2,t_2), \dots, \mu_{r} (q_2,t_2)
\] 
the quantum spectrum of $Y$. Let 
\[
U \subset \operatorname{Spec} (\mathbb{C}[q])
\] 
be a conical open subset chosen as in Section~\ref{subsec:Kstr}.
The quantum spectrum  of ${X \times Y}$ converges to the set 
\[
    \{ \lambda_i (q_1, t_1) + \mu_j (q_2, t_2) \;|\; 0\leq i \leq p, \, 0 \leq j \leq r \}\]
as  $|t| \ll \infty$ and $q$ (and hence also $q_1$ and $q_2$) converges to $\vec{0}$ in $U$. 
\end{theorem}

For this, our key computation will show that the leading terms of the connection matrix arise from the factors \(K_X \otimes \mathrm{id}\) and \(\mathrm{id} \otimes K_Y\), while contributions from mixed curve classes are always of higher order and do not affect the minimal-order terms of the spectrum. This result provides a precise asymptotic description of the spectrum of the quantum connection for products and clarifies the extent to which quantum cohomology of a product decouples at leading order.

Our second main result concerns the structure of the quantum \(D\)-module for product varieties. We construct a formal change of variables that transforms the quantum \(D\)-module of \(X \times Y\) into a form reflecting the factorization into \(X\) and \(Y\). This formal transformation is compatible with the quantum connection and allows us to express flat sections of the product in terms of the flat sections of the factors. As a consequence, questions about solutions of the quantum differential equations and their monodromy reduce, at the level of formal power series, to corresponding questions for \(X\) and \(Y\) separately. For a smooth projective variety $X$, denote by $\mathrm{QDM}(X)^{\mathrm{la}}$ the quantum D-module of $X$ up to a possible formal base change (see Section~\ref{sec:opK} for the details).

\begin{theorem}[{Theorem~\ref{thm:formtensoriso}}]
There exists a a formal isomorphism
\[\mathrm{QDM}(X \times Y)^{\mathrm{la}} \to \mathrm{QDM}(X)^{\mathrm{la}} \otimes \mathrm{QDM}(Y)^{\mathrm{la}} \]
that commutes with the quantum connection.
\end{theorem}

These structural results have two notable applications. First, they allow us to compute \emph{atoms}, birational invariants introduced in \cite{katzarkov2025birational}, for product varieties. By understanding the decomposition of the quantum connection, we can express the atoms of a product in terms of the atoms of the factors, providing new tools for distinguishing birational equivalence classes. Denoting by $\mathrm{Atom}(X)$ the atom of a smooth projective variety $X$, we have the following.
\begin{corollary}[Corollary~\ref{cor:atommult}]
After a suitable formal change of quantum variables, one has
\[
\mathrm{Atom}(X \times Y) \;\sim\; \mathrm{Atom}(X) \otimes \mathrm{Atom}(Y).
\]
\end{corollary}

Second, it is shown in \cite[Section~5.3.2]{katzarkov2025birational} that there exists natural additive map
\[
\begin{aligned}
\phi \colon K_0(\mathrm{Var}) & \longrightarrow \mathbb{Z}^{\bigoplus_{\overline{c}} NSQC_{\overline{c}}}/\sim \\
[X] & \longmapsto \mathrm{Atom}(X),
\end{aligned}
\]
into a certain abelian group $\mathbb{Z}^{\bigoplus_{\overline{c}} NSQC_{\overline{c}}}/\sim$ defined in \cite{kontsevich2025moduli}.
Our results imply the multiplicativity of this map. 

\begin{corollary}[{Corollary~\ref{cor:phiring}}]
Equip the abelian group $\mathbb{Z}^{\bigoplus_{\overline{c}} NSQC_{\overline{c}}}/\sim$ with the multiplication induced by the tensor product of atoms. Then the map $\phi$
is a ring homomorphism.
\end{corollary}
In particular, we get that $\phi$ provides a \emph{motivic measure} as claimed in \cite{katzarkov2025birational}. 

The paper is organized as follows. In Section~\ref{sec:prelim}, we review the necessary background on quantum cohomology, Gromov--Witten invariants, the quantum connection and the tensor product of connections. Section~\ref{sec:gwinvprod} develops the theory of codimension-zero invariants for product varieties, leading to a proof of the decomposition of the leading terms of the quantum connection. Section~\ref{sec:opK} is devoted to the spectral analysis and the formal change-of-variable result for the quantum \(D\)-module. Finally, Section~\ref{sec:applications} presents the applications on quantum atoms and motivic measures.


\subsection*{Acknowledgement}
The author is thankful to Sz. Szabó for helpful discussions on the topic. The author was supported by the János Bolyai Research Scholarship of the Hungarian Academy of Sciences. 

\section{Preliminaries}
\label{sec:prelim}

\subsection{Gromov--Witten invariants}
Let $X$ be a nonsingular projective variety over $\CC$, and let $\beta \in H_2(X,\mathbb{Q})$ be a curve class. For a positive integer $k$, denote by
\[
\overline{\mathcal{M}}_k(X,\beta)
\]
the moduli space of genus-zero, $k$-pointed stable maps
$f \colon (C, p_1, \dots, p_k) \to X$
such that $f_*[C] = \beta$. This is a proper Deligne--Mumford stack equipped with a perfect obstruction theory \cite{behrend1997intrinsic}. The associated virtual fundamental class is denoted
\[
[\overline{\mathcal{M}}_k(X,\beta)]^{\mathrm{vir}}.
\]

Each marked point determines an evaluation map
\[
\mathrm{ev}_i \colon \overline{\mathcal{M}}_k(X,\beta) \to X, 
\qquad 1 \leq i \leq k.
\]
Given cohomology classes $\gamma_1,\dots,\gamma_k \in H^*(X)$, the corresponding genus-zero Gromov--Witten invariant is defined by
\[
I_{\beta}(\gamma_1,\dots,\gamma_k) 
= \int_{[\overline{\mathcal{M}}_k(X,\beta)]^{\mathrm{vir}}}
\prod_{i=1}^k \mathrm{ev}_i^*(\gamma_i).
\]

More generally, it is useful to consider \emph{Gromov--Witten classes} valued in the cohomology of $\overline{\mathcal{M}}_k$, rather than numerical invariants. Following \cite[Section~2]{kontsevich1994gromov}, one proceeds as follows. Let
\[
\eta \colon \overline{\mathcal{M}}_k(X,\beta) \to \overline{\mathcal{M}}_k
\]
be the forgetful morphism. For $\gamma_1,\dots,\gamma_k \in H^*(X)$, set
\[
I^{X}_{\beta}(\gamma_1,\dots,\gamma_k) 
\coloneqq \eta_*\big( \mathrm{ev}_1^*(\gamma_1)\cdots \mathrm{ev}_k^*(\gamma_k) \big) 
\in H^*(\overline{\mathcal{M}}_k).
\]
These are sometimes called \emph{tree-level Gromov--Witten classes}. This gives, for each $k \geq 3$, a map
\[ 
\begin{array}{r c l}
  I^{X}_{\beta}: H^{\ast}(X)^{\otimes k}& \to    & H^{\ast}(\overline{\mathcal{M}}_k) \\
\end{array}
\]
Explicit construction of these classes is given in \cite{behrend1996stacks}.

\begin{remark}
The degree of $I^{X}_{\beta}(\gamma_1,\dots,\gamma_k)$ is determined by the virtual dimension formula. Indeed, the maps $I^{X}_{\beta}$ are homogeneous of degree $2(K_X.\beta)  - 2\dim_{\CC} X$ so
\[
\deg I^{X}_{\beta}(\gamma_1,\dots,\gamma_k) 
= \sum_{i=1}^k \deg \gamma_i + 2(K_X \cdot \beta) - 2\dim_\CC X.
\]
Since $\dim_\CC \overline{\mathcal{M}}_k = k-3$, the codimension of the resulting class is
\[
2k - 6 - \deg I^{X}_{\beta}(\gamma_1,\dots,\gamma_k).
\]
The numerical invariants $I_{\beta}(\gamma_1,\dots,\gamma_k)$ arise exactly from the codimension-zero part of these classes.
\end{remark}

Finally, let $B \subset H_2(X,\mathbb{Q})$ denote the effective cone, i.e.~the cone generated by algebraic curve classes with nonnegative rational coefficients. It is standard that $I_{\beta}(\gamma_1,\dots,\gamma_k) = 0$ for $\beta \notin B$, since in this case the moduli space $\overline{\mathcal{M}}_k(X,\beta)$ is empty.


\subsection{The GW potential and the quantum product}

Fix a basis of the cohomology ring $H^*(X,\mathbb{Q})$ as follows. Let
\[
T_0 = 1 \in H^0(X,\mathbb{Q}), \qquad 
T_1,\dots,T_m \in H^2(X,\mathbb{Q}),
\]
be a basis of $H^2(X,\mathbb{Q})$, and extend this to a full homogeneous basis woth some $T_{m+1},\dots,T_p$ of $H^*(X,\mathbb{Q})$. The corresponding dual basis elements $\{T_i^{\vee}\}$ satisfy the relations $T_i (T_j^{\vee}) = T_j^{\vee} (T_i) = \delta_{ij}$.

Introduce formal variables
\[
t_0, q_1,\dots,q_m, t_{m+1},\dots,t_p,
\]
which we abbreviate collectively as $(q,t)$. The (quantum part of the) genus-zero \emph{quantum potential} of $X$ is the formal Laurent series
\[
\Gamma(q,t) 
= \sum_{\substack{n_{m+1},\dots,n_p \geq 0 \\ \beta \in B \setminus \{0\}}}
I_\beta\!\big( T_{m+1}^{n_{m+1}}\cdots T_p^{n_p} \big) 
\, q_1^{\int_\beta T_1} \cdots q_m^{\int_\beta T_m}
\frac{t_{m+1}^{n_{m+1}}\cdots t_p^{n_p}}{n_{m+1}!\cdots n_p!},
\]
which lives in $\QQ \llbracket q,t \rrbracket[q^{-1}]$. 

\begin{remark}
In general, negative powers of the variables $q_i$ may appear in $\Gamma$; see e.g.~\cite{gyenge2025blow}. It is expected that such negative exponents are bounded below, ensuring that the series belongs to the specified ring. We assume this boundedness property throughout.
\end{remark}

Define differential operators
\[
\partial_i \coloneqq 
\begin{cases}
q_i \dfrac{\partial}{\partial q_i}, & 1 \leq i \leq m, \\
\dfrac{\partial}{\partial t_i}, & i=0, m+1,\dots,p.
\end{cases}
\]
For convenience, set
\[
\Gamma_{ijk} \coloneqq \partial_i \partial_j \partial_k \Gamma.
\]

\begin{definition}\label{def:quantum_product}
The \emph{quantum product} of basis elements $T_i$ and $T_j$ is defined by
\[ T_i \ast_X T_j = T_i \cdot T_j + \sum_{e,f} \Gamma_{ije} g^{ef} T_f, \]
where $g_{ef} = \int_X T_e \cup T_f$ is the Poincaré pairing, and $(g^{ef})$ its inverse.

\end{definition}

When there is no confusion about the base variety, we will just write $T_i \ast T_j$ for this product.
The quantum product endows the free $\mathbb{Q} \llbracket q,t \rrbracket [q^{-1}]$-module generated by $\{T_0, \dots, T_p\}$ with a supercommutative, associateive $\mathbb{Q} \llbracket q,t \rrbracket [q^{-1}]$-algebra structure with unit $T_0$.

\subsection{The Dubrovin connection}
\label{subsec:dubrovin}

Dubrovin \cite{dubrovin1996geometry} showed that the quantum product induces a meromorphic connection and in turn a Frobenius manifold structure on $H^*(X, \mathbb{C})$, with the cup product as nonsingular pairing.

Let
\[ \tau = t_0(\tau)T_0 + \sum_{i=1}^m q_i(\tau)T_i + \sum_{i=m+1}^pt_{i}(\tau)T_i \in H^{\ast}(X, \mathbb{C}) \]
be a cohomology class.

\begin{definition}\label{ref:Dubrovin_connection}
The \emph{Dubrovin} (or \emph{quantum}) \emph{connection} on the trivial bundle
\[ 
H^{\ast}(X,\mathbb{Q}) \times \Spec \mathbb{C}[u, u^{-1}] \to \Spec \mathbb{C}[u, u^{-1}] \]
is the meromorphic flat connection
\begin{equation} 
\label{eq:dubconn1}
\nabla_{\frac{\partial}{\partial u}}^{(\tau)} = \frac{\partial}{\partial u} + \frac{1}{u^2}K^{(\tau)} + \frac{1}{u}G. \end{equation}
Here, $K^{(\tau)}$ is the quantum multiplication operator (see Definition~\ref{def:quantum_product}) with the Euler vector field
\[ c_1(T_X) + \sum_{i=0}^p(2-\deg\,T_i) t_iT_i, \]
and the grading operator $G$ acts on $H^d(X)$ as
\[ G|_{H^d(X)} = \frac{d-2}{2} \cdot \mathrm{Id}_{H^d(X)}. \]
\end{definition}

Suppose that the potential $\Gamma$ converges in $q,t$ or, equivalently, in $\tau$ to an analytic function on some domain. The family of meromorphic connections $\nabla^{(\tau)}$ over $H^{\ast}(X, \mathbb{C})$ then forms a flat connection
$ \nabla \text{ over } \Spec \mathbb{C}\{ q,t \}[u^{\pm}, q^{-1}]$, 
which is an isomonodromic family \cite{dubrovin1998geometry}.

Consider the trivial bundle
\[ H^{\ast}(X,\mathbb{C}) \times \Spec \mathbb{C}\{ q,t \}[q^{-1}, u, u^{-1}] \to \Spec \mathbb{C}\{ q,t \}[q^{-1}, u, u^{-1}], \]
with fibers $H^{\ast}(X,\mathbb{C})$. The coordinates $q, t, u$ provide a full coordinate system on the base, so $\frac{\partial}{\partial q}$, $\frac{\partial}{\partial t}$, and $\frac{\partial}{\partial u}$ form a global basis of the tangent bundle. The extended connection is given by:

\begin{equation}
\label{eq:dubconn2}
\nabla_{\frac{\partial}{\partial t_i}} = \frac{\partial}{\partial t_i} + \frac{1}{u}A_i,
\end{equation}
\begin{equation}
\label{eq:dubconn3}
\nabla_{\frac{\partial}{\partial q}} = \frac{\partial}{\partial q} + \frac{1}{uq}A,
\end{equation}
where $A_i$ is quantum multiplication by $T_i$, and $A$ is quantum multiplication by $c_1(\mathcal{O}(1))$ (with some $i$ satisfying $T_i = c_1(\mathcal{O}(1))$).

The flatness of $\nabla$ is equivalent to the associativity of the quantum product $*$. Equivalently, associativity corresponds to a partial differential equation satisfied by $F$, known as the Witten--Dijkgraaf--Verlinde--Verlinde (WDVV) equation \cite{kontsevich1994gromov, crauder1995quantum}.

\subsection{Tensor product of connections}

Let $E \to M$ and $F \to M$ be vector bundles over the same base $M$ with connections
\[
\nabla^E : \Gamma(E) \to \Gamma(T^*M \otimes E),
\]
\[
\nabla^F : \Gamma(F) \to \Gamma(T^*M \otimes F).
\]
The \emph{tensor product connection} on $E \otimes F$ is given by
\[
\nabla^{E \otimes F}(s \otimes t)
= (\nabla^E s) \otimes t + s \otimes (\nabla^F t),
\]
for $s \in \Gamma(E)$ and $t \in \Gamma(F)$.

\begin{lemma}
Let $E \to M$ be a rank $m$ complex vector bundle and $F \to M$ a rank $n$ complex vector bundle, equipped with connections $\nabla^E$ and $\nabla^F$.  
Choose local frames $(e_1,\dots,e_m)$ for $E$ and $(f_1,\dots,f_n)$ for $F$, and write
\[
\nabla^E e_i = \sum_{k=1}^m A_{ki} \, e_k, 
\quad
\nabla^F f_j = \sum_{\ell=1}^n B_{\ell j} \, f_\ell,
\]
where $A \in \Omega^1(U; M_m(\mathbb{C}))$ and $B \in \Omega^1(U; M_n(\mathbb{C}))$ are the local connection matrices.  
Then, in the induced local frame $(e_i \otimes f_j)_{1 \le i \le m,\ 1 \le j \le n}$ of $E \otimes F$, the tensor product connection
\[
\nabla^{E \otimes F}(s \otimes t) 
= (\nabla^E s) \otimes t + s \otimes (\nabla^F t)
\]
is given by
\[
\nabla^{E \otimes F} = d + (A \otimes I_n) + (I_m \otimes B).
\]
\end{lemma}

\begin{proof}
Applying the definition of $\nabla^{E \otimes F}$ to a basis element $e_i \otimes f_j$ gives
\[
\nabla^{E \otimes F}(e_i \otimes f_j)
= (\nabla^E e_i) \otimes f_j + e_i \otimes (\nabla^F f_j).
\]
Using the local expressions for $\nabla^E e_i$ and $\nabla^F f_j$,
\[
(\nabla^E e_i) \otimes f_j
= \sum_{k=1}^m A_{ki} \, (e_k \otimes f_j)
\quad\text{and}\quad
e_i \otimes (\nabla^F f_j)
= \sum_{\ell=1}^n B_{\ell j} \, (e_i \otimes f_\ell).
\]
The first term corresponds to $A$ acting on the $E$-index and leaving the $F$-index fixed, i.e.\ $A \otimes I_n$, while the second term corresponds to $B$ acting on the $F$-index and leaving the $E$-index fixed, i.e.\ $I_m \otimes B$.  
Thus the total connection matrix in the tensor product frame is
\[
A \otimes I_n + I_m \otimes B,
\]
and adding the exterior derivative $d$ gives the stated formula.
\end{proof}

The sum $(A \otimes I_n) + (I_m \otimes B)$ is called the \emph{Kronecker sum} of the matrices $A$ and $B$. If $R^E$ and $R^F$ are the curvature $2$-forms of $\nabla^E$ and $\nabla^F$, then one can directly check that
\[
R^{E \otimes F} = R^E \otimes I + I \otimes R^F,
\]
again matching the matrix Kronecker sum pattern.

More generally, one can consider tensor product of connections on principal bundles. Let $P_E \to M$ be the principal $GL(V)$-bundle associated to $E$, and  
let $P_F \to M$ be the principal $GL(W)$-bundle associated to $F$.  
The connection $\nabla^E$ corresponds to a $\mathfrak{gl}(V)$-valued connection $1$-form
\[
\omega_E \in \Omega^1(P_E; \mathfrak{gl}(V)),
\]
and similarly $\nabla^F$ corresponds to
\[
\omega_F \in \Omega^1(P_F; \mathfrak{gl}(W)).
\]

The tensor product bundle $E \otimes F$ corresponds to the principal bundle
\[
P_{E \otimes F} := P_E \times_M P_F
\]
with structure group $GL(V \otimes W)$. There is a natural Lie algebra embedding
\[
\iota : \mathfrak{gl}(V) \oplus \mathfrak{gl}(W) \longrightarrow \mathfrak{gl}(V \otimes W),
\]
given by
\[
\iota(A,B) = A \otimes I_W + I_V \otimes B.
\]

\begin{definition}
The \emph{tensor product connection} on $P_{E \otimes F}$ is the connection whose connection $1$-form is
\[
\omega_{E \otimes F} := \iota(\pi_E^* \omega_E, \pi_F^* \omega_F),
\]
where $\pi_E$ and $\pi_F$ are the projections $P_E \times_M P_F \to P_E$ and $P_F$ respectively.
\end{definition}

In this form, the tensor product connection is entirely coordinate-free, and the Kronecker sum
\[
A \otimes I + I \otimes B
\]
appears as soon as one chooses local trivializations of $P_E$ and $P_F$.

\section{GW invariants of product varieties}
\label{sec:gwinvprod}
\subsection{Products}

Let $X, Y$ be nonsingular projective varieties. Assume first that $H^1(X)=H^1(Y)=0$. We introduce the following notation:

\begin{itemize}
\item \( QH^*(X) \) and \( QH^*(Y) \) denote the quantum cohomology rings of $X$ and $Y$;
\item \( K_X \) and \( K_Y \) denote the matrices of quantum multiplication by \( c_1(X) \) and \( c_1(Y) \);
\item \( H^*(X) \) and \( H^*(Y) \) denote the classical cohomology rings of \( X \) and \( Y \);
\item \( \{T_i\}_{i=1}^p \) and \( \{S_j\}_{j=1}^r \) are bases of \( H^*(X) \) and \( H^*(Y) \), respectively;
\item \( \{T_i\}_{i=1}^l \subset H^2(X) \) and \( \{S_j\}_{j=1}^m \subset H^2(Y) \) are bases of $H^2$.
\end{itemize}

Then:
\begin{itemize}
\item \( H^*(X \times Y) \cong H^*(X) \otimes H^*(Y) \), with basis \( \{ T_i \otimes S_j \}_{i,j} \);
\item $H^2(X \times Y)$ has basis $\{T_i \otimes 1\}_{i=1}^l \cup \{1 \otimes S_j\}_{j=1}^m$.
\end{itemize}

In what follows, we identify cohomology classes in $X$ and $Y$ with their pullbacks to $X \times Y$. Concretely, any class $b \in H^2(X \times Y)$ decomposes as
	\[
	b=b_X \otimes 1 + 1 \otimes b_Y, \qquad b_X \in H^2(X), \ b_Y \in H^2(Y);
	\]
	we abbreviate this simply as $\beta=(\beta_X,\beta_Y)$.
	In particular,
	\[
	c_1(X \times Y) \;=\; c_1(X) \otimes 1 + 1 \otimes c_1(Y) \;=\; (c_1(X), c_1(Y)).
	\]
	


The following product formula for Gromov--Witten invariants is due to Behrend.

\begin{theorem}[{\cite{behrend1997product}}]
\label{thm:behrend_prod}
For any curve class $\beta \in H_2(X \times Y)$ and insertions $\gamma_i \in H^*(X)$, $\epsilon_i \in H^*(Y)$,
\[
I^{X \times Y}_{\beta}\bigl((\gamma_1 \otimes \epsilon_1)\cdots (\gamma_n \otimes \epsilon_n)\bigr) 
= (-1)^s \, I^{X}_{\beta_X}(\gamma_1 \cdots \gamma_n)\,\cup\, I^{Y}_{\beta_Y}(\epsilon_1 \cdots \epsilon_n),
\]
where $\beta_X = (p_X)_{\ast}(\beta)$, $\beta_Y = (p_Y)_{\ast}(\beta)$ are the projections of $\beta$, and 
\[
s=\sum_{i>j} \deg \gamma_i \cdot \deg \epsilon_j.
\]
\end{theorem}

\begin{remark}
	If both $H^1(X)$ and $H^1(Y)$ are nontrivial, then also classes from $H^1(X) \otimes H^1(Y)$ appear in the Künneth decomposition of $H^2(X \times Y)$. These have a nontrivial integral only over curve classes $\beta$ for which $ \beta_X \in H_1(X)$ and $ \beta_Y \in H_1(Y)$. But $I^{X}_{\beta_X}(\gamma_1 \cdots \gamma_n)=0$ and $I^{Y}_{\beta_Y}(\epsilon_1 \cdots \epsilon_n)=0$ for such classes. Hence, our results below hold even without the assumption that $H^1(X)=H^1(Y)=0$.
\end{remark}

\begin{lemma} 
\[
\vdim \overline{\mathcal{M}}(X \times Y, \beta) 
= \vdim \overline{\mathcal{M}}(X, \beta_X) \;+\; \vdim \overline{\mathcal{M}}(Y, \beta_Y) \;+\; 3.
\]
\end{lemma}

\begin{proof}
Using the general formula
\[
\vdim \overline{\mathcal{M}}(Z, \beta) = \int_{\beta} c_1(Z) + \dim Z - 3,
\]
we compute
\[
\vdim \overline{\mathcal{M}}(X \times Y, \beta) 
= \int_{\beta} c_1(X \times Y) + \dim(X \times Y) -3.
\]
The additivity of $c_1$ and of the dimension gives the claim.
\end{proof}

We introduce formal vector variables as follows:
\begin{itemize}
\item $q_1$ corresponding to $\{T_i \otimes 1\}_{i=1}^l$ (i.e.\ to $\{T_i\}_{i=1}^l \subset H^2(X)$);
\item $q_2$ corresponding to $\{1 \otimes S_j\}_{j=1}^m$ (i.e.\ to $\{S_j\}_{j=1}^m \subset H^2(Y)$);
\item $t_1$ corresponding to $\{T_i \otimes 1\}_{i \in \{0,l+1,\dots,p\}}$ (i.e.\ to classes in $H^*(X)$ outside $H^2$);
\item $t_2$ corresponding to $\{1 \otimes S_j\}_{j \in \{0,m+1,\dots,r\}}$;
\item $t_{ij}$ corresponding to $T_i \otimes S_j$.
\end{itemize}

For derivatives, we use:
\begin{itemize}
\item $\partial_i$ for $\partial_{T_i \otimes 1}$ or equivalently $\partial_{T_i}$,
\item $\partial_j$ for $\partial_{1 \otimes S_j}$ or equivalently $\partial_{S_j}$,
\item $\partial_{ij}$ for $\partial_{T_i \otimes S_j}$.
\end{itemize}

For index sets, we write:
\begin{itemize}
\item $P = \{0,\dots,p\}$, $R = \{0,\dots,r\}$,
\item $L=\{1,\dots,l\}$, $M=\{1,\dots,m\}$.
\end{itemize}

Define $v_X(\beta_X)= \vdim\overline{\mathcal{M}}(X, \beta_X)+3$ and $v_Y(\beta_Y)= \vdim\overline{\mathcal{M}}(Y, \beta_Y)+3$. For bookkeeping of mixed insertions, let
\[
t^v = \sum_{ \substack{ (i,j) \in P \times R \setminus (L \cup M)\\ \sum_{i,j} v_{ij} \cdot \mathrm{codim}( T_{i}\otimes S_{j}) = v }}  
\prod_{i,j}\frac{t_{ij}^{v_{ij}}}{v_{ij}!},
\]
where
\[
\mathrm{codim}( T_{i}\otimes S_{j}) = 2(\dim X+\dim Y) - \deg T_i - \deg S_j.
\]

\begin{proposition} 
\label{prop:proddecomp}
The potential $\Gamma_{X \times Y}(q,t)$ satisfies
\[
\Gamma_{X \times Y}(q,t) 
= \Gamma_X(q_1,t_1)\, t^{v_Y(0)} \;+\; \Gamma_Y(q_2,t_2)\, t^{v_X(0)} \;+\; \text{mixed terms},
\]
where the ``mixed terms'' arise from summing over curve classes
\[
(\beta_X, \beta_Y) \in H_2(X) \oplus H_2(Y) \setminus (H_2(X) \cup H_2(Y)),
\]
i.e.\ classes where both components are nonzero.
\end{proposition}

\begin{proof}
Recall the string equation (or fundamental class axiom) for Gromov--Witten invariants:
\[
I_{\beta}(\gamma_1\cdots\gamma_k \cdot 1)=0
\quad \text{unless $\beta=0$ and $k=2$.}
\]
Any variable $q_1$, resp. $q_2$, in $X \times Y$ corresponds to a divisor class of the form $T_i \otimes 1$, resp. $1 \otimes S_j$. By the string equation and the product formula (Theorem~\ref{thm:behrend_prod}), a Gromov--Witten invariant involving such a class vanishes unless $\beta_X$, resp. $\beta_Y$, is zero. This proves the decomposition.
\end{proof}

\begin{example}
Let $X = Y = \mathbb{P}^1$. Then
\[
\Gamma_X(q,t)=\Gamma_Y(q,t) = q,
\]
and
\[
\Gamma_{X \times Y}(q,t)
= q_1 t_p + q_2 t_p + \frac{q_1q_2t_p^3}{6}
+ \frac{q_1^2q_2t_p^5}{120}
+ \frac{q_1q_2^2t_p^5}{120}
+ \dots
\]
Here the product $q_1t_p \cdot q_2t_p$ corresponds to a class of codimension two, hence not a valid Gromov--Witten invariant. By contrast, the term $q_1q_2t_p^3$ has codimension zero and thus contributes.
\end{example}

\subsection{Obtaining codimension zero invariants}

\begin{proposition} 
\label{prop:firstderiv}
Let $i \in \{0,\dots,p\}$ and $j \in \{0,\dots,r\}$. Suppose the invariant corresponding to $(\beta,\gamma)$ is nonzero and contributes a term to the derivative $\partial_{ij} \Gamma_{X\times Y}$. Then there exists a term in $\partial_{i} \Gamma_X$ for some $\beta'_X \leq \beta_X$. Similarly, there exists a term in $\partial_{j}\Gamma_Y$ for some $\beta'_Y \leq \beta_Y$. 
\end{proposition}

\begin{proof}
We apply Theorem~\ref{thm:behrend_prod}, which describes how invariants of $X \times Y$ decompose into products of invariants from $X$ and from $Y$. If the invariant on the left-hand side is nonzero, then the corresponding invariants on the right-hand side must also be nonzero. Consequently, there are nonzero summands in the potentials of $X$ and $Y$ contributing to $\Gamma_X$ and $\Gamma_Y$, respectively. 

Consider the $X$-part, say $I_{\beta_X}^X(\gamma_1,\dots,\gamma_n)$. This class need not be of codimension zero. We can assume without loss of generality the the classes $\gamma_i$, $1 \leq i \leq n$ are homogeneous.

Recall that for every decomposition $I \sqcup J = \{1,\dots,n\}$ there is a gluing map
\[
\varphi_S: \overline{\mathcal{M}}_{|I|+1} \times \overline{\mathcal{M}}_{|J|+1} \to \overline{\mathcal{M}}_{n}.
\]
The splitting axiom for tree-level Gromov--Witten classes \cite[6.1.2]{kontsevich1994gromov} states that
\begin{equation}
\label{eq:splitting_axiom}
\varphi_S^{\ast}(I_{\beta}(\gamma_1,\dots,\gamma_n)) = 
\sum_{\beta_1 + \beta_2 = \beta} 
\sum_{e,f} 
I_{\beta_1}( \{ \gamma_c \}_{c \in I}, T_e ) \, g^{ef} \, I_{\beta_2}( T_f, \{ \gamma_d \}_{d \in J} ).
\end{equation}
Iterating this decomposition, we eventually obtain a nonzero Gromov--Witten invariant $I_{\beta'_X}(\gamma')$ of codimension zero with $\beta'_X \leq \beta_X$. Moreover, by construction, one of the insertions couples nontrivially with $T_i^\vee$. If $i \in \{0,l+1,\dots,p\}$ this is achieved through the cohomology class $\gamma'$, while if $1 \leq i \leq l$ it is achieved through the curve class $\beta'_X$. Namely, $T_i^{\vee}(\gamma') \neq 0$ if $i \in \{0,l+1,\dots,p\}$ or $T_i^{\vee}(\beta'_X) \neq 0$ if $1 \leq i \leq l$.
The existence of such a pair $(\beta'_X,\gamma')$ follows from the recursive reduction procedure leading to codimension zero invariants \cite[Proposition~2.5.2]{kontsevich1994gromov} (see also the Second Reconstruction Theorem). Thus $\partial_i \Gamma_X$ has a corresponding nonzero summand.

The argument for the $Y$-part is identical.
\end{proof}


To strengthen this analysis we need two auxiliary lemmas.
\begin{lemma}
\label{lem:effective_triv_int}
Let \( X \) be a smooth variety, and let \( D = \sum_{i=1}^n a_i D_i \) be an effective Cartier divisor on \( X \), where each \( D_i \) is an irreducible effective Cartier divisor and \( a_i > 0 \). Let \( V \subset X \) be a subvariety that intersects each \( D_i \) properly. If the intersection product \( V \cdot D = 0 \) in the Chow group \( A_*(X) \), then 
\[
V \cdot D_i = 0 \quad \text{in } A_*(X) \quad \text{for all } i.
\]
\end{lemma}

\begin{proof}
Since \( V \) intersects each \( D_i \) properly and each \( D_i \) is a Cartier divisor, the intersection product \( V \cdot D_i \) is well-defined in the Chow group and satisfies:
\[
V \cdot D = \sum_{i=1}^n a_i (V \cdot D_i).
\]
Because \( D \) is effective with \( a_i > 0 \), and \( V \cdot D_i \) is an effective cycle (or zero) due to proper intersection and effectiveness, each term \( V \cdot D_i \) lies in the monoid of effective cycles. Therefore, if 
\[
V \cdot D = 0,
\]
and all summands are non-negative, then each must vanish:
\[
V \cdot D_i = 0 \quad \text{for all } i.
\]
\end{proof}

\begin{lemma} 
\label{lem:ijtogether}
Let $I_{\beta}(\gamma_1,\dots,\gamma_n)$ be a nontrivial class of codimension $> 0$ and let $i,j \in \{1,\dots,n\}$. Then there exists a splitting $I \sqcup J = \{1,\dots,n\}$ such that $i,j \in I$ and $\varphi_S^{\ast} (I_{\beta}(\gamma_1,\dots,\gamma_n)) \neq 0$.
\end{lemma}
\begin{proof}
Suppose by contradiction that for all splittings with $i,j \in I$ we have that $\varphi_S^{\ast} (I_{\beta}(\gamma_1,\dots,\gamma_n)) = 0$. Then the intersection of $I_{\beta}(\gamma_1,\dots,\gamma_n)$ with the boundary divisor $D(ij|kl)$ is zero:
\[ \sum_{\substack{I \sqcup J = \{1,\dots,n\} \\ i,j \in I, k,l \in J\\ \beta_1 + \beta_2 = \beta}} 
\sum_{e,f} 
I_{\beta_1}\left( \{ \gamma_c \}_{c \in I}, T_e \right) \, 
g^{ef} \, 
I_{\beta_2}\left( T_f, \{ \gamma_d \}_{d \in J} \right) = 0.\]
But then the intersection with the linearly equivalent divisor $D(ik|jl)$ is also zero
\[
\sum_{\substack{I' \sqcup J' = \{1,\dots,n\} \\ i,k \in I', j,l \in J'\\ \beta'_1 + \beta'_2 = \beta}} 
\sum_{e,f} 
I_{\beta'_1}\left( \{ \gamma_c \}_{c \in I'}, T_e \right) \, 
g^{ef} \, 
I_{\beta'_2}\left( T_f, \{ \gamma_d \}_{d \in J'} \right) =0.
\]
It is known that $I_{\beta}(\gamma_1,\dots,\gamma_n))$ intersects properly $D(ik|jl)$ (see e.g. \cite[Lemma~2.8.2]{kock2007invitation}) and that $D(ik|jl)$ is effective.
Then, by Lemma~\ref{lem:effective_triv_int}, these intersections must be all trivial.
As this is true for all $k,l$, we have that the intersection of $I_{\beta}(\gamma_1,\dots,\gamma_n)$ is trivial with all boundary divisors. Therefore, $I_{\beta}(\gamma_1,\dots,\gamma_n))$ is supported on the open stratum. That forces the codimension to be zero, contradicting the assumption. 

\end{proof}

\begin{proposition}
\label{prop:secondderiv}
Let $i_1,i_2 \in \{0,\dots,p\}$ and $j_1,j_2 \in \{0,\dots,r\}$. Suppose the invariant corresponding to $(\beta,\gamma)$ is nonzero and hence contributes a term to the derivative $\partial_{i_1j_1} \partial_{i_2j_2} \Gamma_{X\times Y}$. 
Then there exists a term in $\partial_{i_1} \partial_{i_2}\Gamma_X$ for some $\beta'_X \leq \beta_X$. Similarly, there exists a term in $\partial_{j_1}\partial_{j_2} \Gamma_Y$ for some $\beta'_Y \leq \beta_Y$. 
\end{proposition}

\begin{proof}
We again focus only on the $X$ part; the argument for the $Y$-part is the same.  

Assume first that $T_{i_2}$ is not a divisor class. Let $i, j$ be indices such that $T_{i_1}^{\vee}(\gamma_i)\neq 0$ and $T_{i_2}^{\vee}(\gamma_j)\neq 0$.
The recursion process invoked in the proof of Proposition~\ref{prop:firstderiv} partitions the index set $\{1,\dots,n\}$. We can think about it as a tree (graph) with a partition in each of the nodes and codimension zero classes in the leaves. 
By Lemma~\ref{lem:ijtogether} we can ensure that at each step $i$ and $j$ end up in the same part.


Suppose now that $T_{i_2}$ is a divisor class. 
Using the divisor axiom \cite[2.2.4]{kontsevich1994gromov}, we replace $I_{\beta}(\gamma_1,\dots,\gamma_n)$ with 
\[ \frac{1}{\int_{\beta}T_{i_2}} \cdot I_{\beta}(\gamma_1,\dots,\gamma_n,T_{i_2}).\] 
We can do this because $\int_{\beta}T_{i_2} \neq 0$ by the assumptions.
By the above argument, we can enforce that $T_{i_2}$ stays in the same component as $\gamma_i$.

\end{proof}

\begin{lemma}
\label{lem:codimzeroglue}
Let \( X \) be a smooth projective variety of dimension \( d \). Let \( \gamma_1, \dots, \gamma_n, \gamma'_1, \dots, \gamma'_m \in H^*(X, \mathbb{Q}) \) be cohomology classes. Let \( \{T_e\} \) and \( \{T^e\} \) be dual bases of \( H^*(X, \mathbb{Q}) \) with respect to the Poincar\'e pairing, so that \( \langle T_e, T^f \rangle = \delta_e^f \). Let moreover \( \beta_1, \beta_2 \in H_2(X, \mathbb{Z}) \) be arbitrary curve classes. If the two Gromov--Witten classes
\[
I_{\beta_1}(\gamma_1, \dots, \gamma_n, T_e) \in H^*(\overline{M}_{0,n+1}(X,\beta_1), \mathbb{Q})
\]
and
\[
I_{\beta_2}(T^e, \gamma'_1, \dots, \gamma'_m) \in H^*(\overline{M}_{0,m+1}(X,\beta_2), \mathbb{Q}).
\]
are nontrivial, then their fiber product over the evaluation maps at the insertions \( T_e \) and \( T^e \) yields a nonzero contribution to the class
\[
I_{\beta_1 + \beta_2}(\gamma_1, \dots, \gamma_n, \gamma'_1, \dots, \gamma'_m) \in H^*(\overline{M}_{0,n+m}(X, \beta_1 + \beta_2), \mathbb{Q})
\]
of codimension \( d = \dim X \). By inserting additional cohomology classes of total degree \( d \), one can promote the resulting class to codimension zero.
\end{lemma}

\begin{proof}
The classes \( I_{\beta_1}(\gamma_1, \dots, \gamma_n, T_e) \) and \( I_{\beta_2}(T^e, \gamma'_1, \dots, \gamma'_m) \) live on the moduli spaces \( \overline{M}_{0,n+1}(X,\beta_1) \) and \( \overline{M}_{0,m+1}(X,\beta_2) \), respectively. The insertions \( T_e \) and \( T^e \) are paired via evaluation maps at the last marked points and the Poincar\'e pairing on \( H^*(X) \).

Gluing these two moduli spaces along the evaluation maps defines a boundary stratum of \( \overline{M}_{0,n+m}(X, \beta_1 + \beta_2) \), and the corresponding contribution is given by the fiber product over \( X \). Since the diagonal in \( X \times X \) has codimension \( d \), the resulting pushforward to \( \overline{M}_{0,n+m}(X, \beta_1 + \beta_2) \) increases the codimension by \( d \).

Therefore, the glued class lies in codimension \( d \) relative to the virtual dimension of \( \overline{M}_{0,n+m}(X, \beta_1 + \beta_2) \). This contribution is nonzero if the individual Gromov--Witten classes and the pairing \( \langle T_e, T^e \rangle \) are nonzero.

To obtain a codimension zero (numerical) invariant, one may insert additional cohomology classes whose total degree equals \( d \), thereby compensating for the excess codimension introduced by gluing. An example which works in any setting is $T_p$, the dual cohomology class of a point.
\end{proof}

\begin{proposition}
Let $i_1,i_2,i_3 \in \{0,\dots,p\}$ and $j_1,j_2,j_3 \in \{0,\dots,r\}$. Suppose the invariant corresponding to $(\beta,\gamma)$ is nonzero and contributes a term to the derivative $\partial_{i_1j_1} \partial_{i_2j_2} \partial_{i_3j_3}\Gamma_{X\times Y}$. 
Then there exists a term in $\partial_{i_1} \partial_{i_2}\partial_{i_3}\Gamma_X$ for some $\beta'_X \leq \beta_X$. Similarly, there exist a term in $\partial_{j_1}\partial_{j_2}\partial_{j_3} \Gamma_Y$ for some $\beta'_Y \leq \beta_Y$. 
\end{proposition}
\begin{proof}
Let $u, v, w$ be indices such that $T_{i_1}^{\vee}(\gamma_u)\neq 0$, $T_{i_2}^{\vee}(\gamma_v)\neq 0$ and $T_{i_3}^{\vee}(\gamma_w)\neq 0$. In general, it is not guaranteed that the recursive algorithm as tweaked in Proposition~\ref{prop:secondderiv} gives a codimension zero invariant which contains $\gamma_u$, $\gamma_v$ and $\gamma_w$ simultaneously. But it always give a collection of invariants, from which we can select a sequence
\[
I_{\beta_1}(\gamma_u,\gamma_v,\dots,T_e), \quad 
I_{\beta_2}(T^e,\dots,T_f), \quad 
\dots, \quad
I_{\beta_s}(T^g,\dots,\gamma_w)
\]
where each consecutive pair involves a dual pair $(T_e,T^e)$, ensuring compatibility under gluing. Each of these classes is of codimension zero. By Lemma~\ref{lem:codimzeroglue}, we may glue these contributions inductively to obtain a nonzero codimension zero invariant involving all three insertions, yielding the required summand.
\end{proof}

\begin{corollary} 
\label{cor:noleadingmixed}
For every triple of indices $(a,b,c)$, no leading term of $\Gamma_{abc}^{X \times Y}$ arises from mixed terms. 
\end{corollary}


\section{Quantum D-modules}
\label{sec:opK}

\subsection{Structure of $K$}
\label{subsec:Kstr}

We now investigate the asymptotic behaviour of the spectrum of the Dubrovin connection when $q$ is in the neighborhood of the origin, with $t$ fixed such that $|t| \ll \infty$. Our focus is on the leading (lowest order) terms in the connection matrix with respect to the $q$-variables, while keeping $t \neq 0$ fixed.

\begin{defn}
The \emph{degree} or, interchangeably, the \emph{order} of a summand in $K_{ij}$ is defined as the sum of the exponents of the $q$-variables in that summand.
If a summand $q^{\beta} t^\gamma$ attains the minimal order among all summands of $K_{ij}$, we say that $(\beta,\gamma)$ is \emph{minimal at $ij$}.
\end{defn}


\begin{proposition}
\label{prop:prod_euler_lead}
Let $X$ and $Y$ be smooth projective varieties, and let $K_X$, $K_Y$ denote the quantum multiplication operators with their respective Euler vector fields. Then the quantum multiplication operator $K_{X \times Y}$ with the Euler vector field on $X \times Y$ satisfies
\[
K_{X \times Y} = K_X \otimes \mathrm{id} + \mathrm{id} \otimes K_Y + \text{(higher order terms)}.
\]
In particular, the minimal order terms of $K_{ij}$ arise from summands of $(K_X \otimes \mathrm{id})_{ij}$ or $(\mathrm{id} \otimes K_Y)_{ij}$, for $(i,j) \in P \times R$.
\end{proposition}

\begin{proof}
From Proposition~\ref{prop:proddecomp}, we have
\[
(K_X \otimes 1) \ast (\alpha_1 \otimes \alpha_2) =
(K_X \ast_X \alpha_1) \otimes \alpha_2 + \text{(mixed terms)},
\]
and similarly
\[
(1 \otimes K_Y) \ast (\alpha_1 \otimes \alpha_2) =
\alpha_1 \otimes (K_Y \ast_Y \alpha_2) + \text{(mixed terms)}.
\]

The mixed terms involve Gromov--Witten invariants of curve classes $(\beta_X, \beta_Y)$ with $\beta_X \ne 0$ and $\beta_Y \ne 0$, so they contribute only at order $q_1^{\beta_X} q_2^{\beta_Y}$. By Corollary~\ref{cor:noleadingmixed}, these contributions are never minimal order terms. Therefore, the leading part of $K_{X \times Y}$ is given by
\[
K_X \ast_X \otimes \mathrm{id} + \mathrm{id} \otimes K_Y \ast_Y,
\]
with higher-order corrections arising from mixed terms, as claimed.
\end{proof}

\begin{lemma}
\label{lem:kronecker_prod_spectrum}
Let \( A \in \mathbb{C}^{m \times m} \) and \( B \in \mathbb{C}^{n \times n} \) have eigenvalues \( \{ \lambda_1, \dots, \lambda_m \} \) and \( \{ \mu_1, \dots, \mu_n \} \), respectively. Then the Kronecker sum
\[
A \otimes I_n + I_m \otimes B
\]
has eigenvalues \( \lambda_i + \mu_j \) for all \( i = 1, \dots, m \) and \( j = 1, \dots, n \).
\end{lemma}

\begin{proof}
Let \( v_i \in \mathbb{C}^m \) be an eigenvector of \( A \) with eigenvalue \( \lambda_i \), and \( w_j \in \mathbb{C}^n \) an eigenvector of \( B \) with eigenvalue \( \mu_j \). Then
\[
A v_i = \lambda_i v_i, \quad B w_j = \mu_j w_j.
\]
The Kronecker product \( v_i \otimes w_j \in \mathbb{C}^{mn} \) satisfies
\begin{align*}
(A \oplus B)(v_i \otimes w_j) &= (A \otimes I_n + I_m \otimes B)(v_i \otimes w_j) \\
&= (A v_i) \otimes w_j + v_i \otimes (B w_j) \\
&= \lambda_i v_i \otimes w_j + v_i \otimes \mu_j w_j \\
&= (\lambda_i + \mu_j)(v_i \otimes w_j).
\end{align*}
Thus, \( v_i \otimes w_j \) is an eigenvector of \( A \oplus B \) with eigenvalue \( \lambda_i + \mu_j \), and the result follows.
\end{proof}

Let 
\[
\lambda_0 (q_1,t_1), \dots, \lambda_{p} (q_1,t_1)
\] 
denote the spectrum of $K_X$, and 
\[
\mu_0 (q_2,t_2), \dots, \mu_{r} (q_2,t_2)
\] 
the spectrum of $K_Y$. Let 
\[
U \subset \operatorname{Spec} (\mathbb{C}[q])
\] 
be a conical open subset in the analytic topology whose closure $\overline{U}$ satisfies
\[
\overline{U} \cap \{ q_1^i = 0 \} = \{ \vec{0} \}, \quad 1 \le i \le l, 
\quad \text{and} \quad
\overline{U} \cap \{ q_2^j = 0 \} = \{ \vec{0} \}, \quad 1 \le j \le m.
\]
We assume $t_1$ and $t_2$ are chosen so that $(q_1,t_1)$, $(q_2,t_2)$, and $(q,t)$ all lie within the domain of convergence of $F_X$, $F_Y$, and $F_{X \times Y}$, respectively. Our first main result is the following.

\begin{theorem}
\label{thm:main1_eiglim}
With the above notations, the set of eigenvalues of the operator $K_{X \times Y}$ converges to the set 
\[
    \{ \lambda_i (q_1, t_1) + \mu_j (q_2, t_2) \;|\; 0\leq i \leq p, \, 0 \leq j \leq r \}\]
as  $|t| \ll \infty$ and $q$ (and hence also $q_1$ and $q_2$) converges to $\vec{0}$ in $U$. 
\end{theorem}

\begin{proof}
By Proposition~\ref{prop:prod_euler_lead}, the leading terms of the characteristic polynomial of $K_{X \times Y}$ come from $K_X$ and $K_Y$ separately. Lemma~\ref{lem:kronecker_prod_spectrum} then shows that at $q=0$ the characteristic polynomial splits into linear factors corresponding to all sums $\lambda_i + \mu_j$. 

Finally, Hensel's lemma \cite[Lecture 12]{abhyankar1990algebraic} ensures that this splitting persists for generic small $q$, giving the stated convergence of eigenvalues. This argument follows the same strategy as in \cite[Theorem~5.3]{gyenge2025blow}.
\end{proof}

\subsection{Base change in quantum cohomology}

In certain calculations one base of the cohomology ring is preferred over another. A base change in cohomology also affects the particular expressions for the quantum multiplication and the Dubrovin connection.

After computing the quantum multiplication matrices with respect to the original basis, we conjugate these matrices by the base change matrix. Importantly, the Novikov variables \( q_i \) also depend on the choice of a basis for the curve classes in homology, so a base change in cohomology must be accompanied by a corresponding change of variables to reflect the new dual basis. This procedure allows us to track how quantum corrections transform under coordinate change and can be used to verify consistency between different presentations or computations.

Let us illustrate this in the case of \( X = \mathbb{P}^1 \times \mathbb{P}^1 \). Denote by \( H_1 \) and \( H_2 \) the pullbacks of the hyperplane classes from the two factors, so that \( H^\ast(X) = \mathbb{C}[H_1,H_2]/(H_1^2, H_2^2) \). Let us consider the standard basis \( \{1, H_1, H_2, pt\} \) of \( H^\ast(X) \), and let \( q_1 = e^{t_1} \), \( q_2 = e^{t_2} \) be the Novikov variables corresponding to the classes dual to \( H_1 \) and \( H_2 \), respectively.

The leading terms of the operator $K$ in this basis were computed in \cite[Example~4.16]{gyenge2025blow}:
\[
\begin{pmatrix}
0 & 2q_1 & 2q_2 & 2q_1q_2t_p \\
2& -q_1q_2\frac{t_p^3}{3} & q_2t_p & 2q_2 \\
2& q_1t_p & -q_1q_2\frac{t_p^3}{3} & 2q_1 \\
 -2 t_p & 2 & 2  & 0
\end{pmatrix}
\]

Now suppose we wish to change to the basis \( \{1, H, D, pt\} \), where \( H = H_1 + H_2 \) and \( D = H_1 - H_2 \). This change of basis is represented by the matrix
\[
T = 
\begin{pmatrix}
1 & 0 & 0 & 0 \\
0 & 1 & 1 & 0 \\
0 & 1 & -1 & 0 \\
0 & 0 & 0 & 1
\end{pmatrix}
\quad\quad
\textrm{with inverse}
\quad\quad
T^{-1} =
\begin{pmatrix}
1 & 0 & 0 & 0 \\
0 & \frac{1}{2} & \frac{1}{2} & 0 \\
0 & \frac{1}{2} & -\frac{1}{2} & 0 \\
0 & 0 & 0 & 1
\end{pmatrix}.
\]
Conjugating by \( T \), we obtain 
\[
T^{-1} \cdot [K] \cdot T =
\begin{pmatrix} 0 & 2 q_{1} + 2 q_{2} & 2 q_{1} - 2 q_{2} & 2 q_{1} q_{2} t_{p} \\ 
2 & - \frac{q_{1} q_{2} t_{p}^{3}}{3}+\frac{q_{1} t_{p}+q_2t_p}{2} & \frac{q_{1} t_{p}-q_2t_p}{2} & q_{1} + q_{2} \\ 
0 & \frac{-q_{1} t_{p}+q_2t_p}{2} & - \frac{q_{1} q_{2} t_{p}^{3}}{3}-\frac{q_{1} t_{p}+q_2t_p}{2} & - q_{1} + q_{2} \\ 
- 2 t_{p} & 4 & 0 & 0 \end{pmatrix}.
\]
Lastly, we have to substitute the Novikov variables as $q_h = q_1q_2$ and $q_d=q_1q_2^{-1}$. Equivalently,
\[q_1= \sqrt{q_1q_2 \cdot q_1q_2^{-1}}=\sqrt{q_hq_d}\] and
\[q_2= \sqrt{q_1q_2 \cdot (q_1q_2^{-1})^{-1}}=\sqrt{q_hq_d^{-1}}.\]
This implies that the leading term of the matrix of $K$ in the new basis \( \{1,H,D,pt\} \) is
\[
\begin{pmatrix} 0 & 2 \sqrt{q_hq_d} + 2 \sqrt{q_hq_d^{-1}} & 2 \sqrt{q_hq_d} - 2 \sqrt{q_hq_d^{-1}} & 2 q_h t_{p} \\ 
2 & - \frac{q_h t_{p}^{3}}{3}+\frac{\sqrt{q_hq_d} t_{p}+\sqrt{q_hq_d^{-1}}t_p}{2} & \frac{\sqrt{q_hq_d} t_{p}-\sqrt{q_hq_d^{-1}}t_p}{2} & \sqrt{q_hq_d} + \sqrt{q_hq_d^{-1}} \\ 
0 & \frac{-\sqrt{q_hq_d} t_{p}+\sqrt{q_hq_d^{-1}}t_p}{2} & - \frac{q_h t_{p}^{3}}{3}-\frac{\sqrt{q_hq_d} t_{p}+\sqrt{q_hq_d^{-1}}t_p}{2} & - \sqrt{q_hq_d} + \sqrt{q_hq_d^{-1}} \\ 
- 2 t_{p} & 4 & 0 & 0 \end{pmatrix}.
\]

\subsection{Formal decomposition}

One can view the quantum product as a formal family of supercommutative product structures $\ast_{\tau}$ on $H^\ast(X)\otimes
\CC[[Q]]$ parametrized by $\tau \in H^{\ast}(X)$. The coefficient ring $\CC[[Q]]$ is the
completed monoid ring of the set $NE_{\NN}(X) \subset H_2(X, \ZZ)$ of effective curves, which is called the \emph{Novikov ring}. 

The quantum D-module $QDM(X)$ is the module $H^{\ast}
(X) \otimes \CC[u][[Q, \tau ]]$
equipped with the mutually supercommuting operators from $QDM(X)$ to
$u^{-2} QDM(X)$ given by the Dubrovin connection \eqref{eq:dubconn1}, \eqref{eq:dubconn2} and \eqref{eq:dubconn3}. 

In fact, the Hukuhara-Levelt-Turrittin Theorem implies that there exists a formal decomposition
\begin{equation}
\label{eq:hltdecomp}
QDM(X) = \bigoplus_{a \in \mathbb{C}} \left( d+\frac{a}{u^2} du \right) \otimes \nabla_{\mathrm{reg}}^a
\end{equation}
where the summation is finite and $\nabla_{\mathrm{reg}}^a$ has a regular singularity at $u=0$ \cite[Page~3]{kontsevich2025moduli}. The set of numbers $a$ appearing non-trivially in this decomposition is called the \emph{quantum spectrum} of the connection.

Let $X$ be a smooth projective variety, and $Z \subset X$ a smooth subvariety of codimension $r \geq 2$. Let $\widetilde{X}$ be the blow-up of $X$ along $Z$, and let $E \subset \widetilde{X}$ be the exceptional divisor. To compare the quantum D-modules of these varieties, one needs to embed their Novikov ring into a common ring. Denote by $Q$, $\widetilde{Q}$, $Q_Z$ and $Q_E$ the respective Novikov variables of $X$, $\widetilde{X}$, $Z$ and $E$. In particular, $Q_E$ is a single variable. The following statement about the behaviour of the quantum D-module with respect to blow-ups, conjectured originally by M. Kontsevich, was proved recently by the author and Sz. Szabó  in the surface case, and in general by H. Iritani.

\begin{theorem}[{\cite{gyenge2025blow}, \cite{iritani2023quantum}}] 
\label{thm:quantumbu}
There exists a formal invertible change of variables
\[ H^{\ast}(\widetilde{X})\oplus H^{\ast}(Z) \to H^{\ast}(X) \oplus H^{\ast}(E),\quad (\tilde{\tau},\tau_Z) \mapsto (\tau (\tilde{\tau},\tau_Z),\tau_E(\tilde{\tau},\tau_Z))
\] defined over $\C((Q_E^{-\frac{1}{r-1}}))[[Q]]$ and an isomorphism
\[\mathrm{QDM}(\widetilde{X})^{\mathrm{la}} \oplus \mathrm{QDM}(Z)^{\mathrm{la}} \to \tau^{\ast}\mathrm{QDM}(X)^{\mathrm{la}} \oplus \tau_E^{\ast}\mathrm{QDM}(E)^{\mathrm{la}}\]
that commutes with the quantum connection.

\end{theorem}
Here the superscript ``$\mathrm{la}$'' refers to a base change of the respective quantum D-module to $\C((Q_E^{-\frac{1}{r-1}}))[[Q]]$ via certain ring extensions. To get the above form from the one stated in \cite{iritani2023quantum} note that $r$ copies of $\mathrm{QDM}(Z)^{\mathrm{la}}$ gives exactly $\mathrm{QDM}(E)^{\mathrm{la}}$. This is because $E$ is a projective bundle over $Z$, so one can apply the quantum Leray--Hirsch theorem \cite[Theorem~0.2]{lee2013invariance}.

In the same vein, for the  product of two varieties $X$ and $Y$ a formal change of variables allows us to get rid of the higher order terms. In this case the Novikov rings get extended to $\C[[Q_1,Q_2]]$ where $Q_1$ is the Novikov (vector)variable of $X$ and $Q_2$ is the Novikov (vector)variable of $Y$.
\begin{theorem} 
\label{thm:formtensoriso}
There exists a a formal invertible change of variables
\[ H^{\ast}(X \times Y) \to H^{\ast}(X) \otimes H^{\ast}(Y),\quad \tau \mapsto \sum_l \tau_X^l \otimes \tau_Y^l
\] defined over $\C[[Q_1,Q_2]]$
and an isomorphism
\[\mathrm{QDM}(X \times Y)^{\mathrm{la}} \to \mathrm{QDM}(X)^{\mathrm{la}} \otimes \mathrm{QDM}(Y)^{\mathrm{la}} \]
that commutes with the quantum connection.
Here $\mathrm{QDM}(X)^{\mathrm{la}} \otimes \mathrm{QDM}(Y)^{\mathrm{la}}$ in the $\partial_u$-direction is
\[ \nabla_{\partial_u} = \partial_u - \frac{1}{u^2}\left(K_X \otimes id + id \otimes K_Y \right) + \frac{1}{u} G_{X \times Y}.\]
\end{theorem}
\begin{proof}
Follows from Proposition~\ref{prop:prod_euler_lead}.
\end{proof}

\begin{corollary}
The quantum spectrum of $X \times Y$ consists of the points $a + b$ where $a$ is a spectrum point of $X$ and $b$ is a spectrum point of $Y$.
\end{corollary}
\begin{proof}
Follows from Proposition~\ref{lem:kronecker_prod_spectrum}.
\end{proof}

\section{Birational and motivic invariants}
\label{sec:applications}

\subsection{Product of atoms}

Consider a vector $\overline{c}=(c_1,\dots,c_r)$ of length $r$ whose nonzero entries $c_i \in \mathbb{C}^{\times}$ are not all equal to one. 
Following \cite{kontsevich2025moduli}, we denote by 
\[
NSQC_{\overline{c}}
\]
the groupoid of \emph{framed meromorphic connections} of rank $r$ with prescribed singularity data determined by $\overline{c}$. These objects arise naturally in the study of quantum connections and their Stokes structures, and may be regarded as the building blocks of the categories underlying non-semisimple quantum cohomology. Passing to isomorphism classes, one obtains the coarse moduli space associated to $NSQC_{\overline{c}}$.

Recently, Katzarkov--Kontsevich--Pantev--Yu introduced in \cite{katzarkov2025birational} the notion of an \emph{atom}. Roughly speaking, given a smooth projective variety $X$, the set 
\[
\mathrm{Atom}(X)
\]
consists of the direct summands of the quantum connection of $X$ appearing in the decomposition \eqref{eq:hltdecomp}, counted with multiplicity and considered modulo the following equivalence relations:

\begin{enumerate}
\item (Disjoint union) For any pair of smooth projective varieties $X_1, X_2$,
\[
\mathrm{Atom}(X_1 \sqcup X_2) \;\sim\; \mathrm{Atom}(X_1) + \mathrm{Atom}(X_2).
\]
\item (Blow-up relation) For a blow-up $\widetilde{X} \to X$ along a smooth center $Z \subset X$ of codimension $r \geq 2$,
\[
\mathrm{Atom}(\widetilde{X}) \;\sim\; \tau^{\ast} \mathrm{Atom}(X) \;+\; \left( \sum_{i=1}^{r-1} \zeta_i^{\ast}\, \mathrm{Atom}(Z) \right),
\]
for suitable formal changes of quantum parameters $\tau, \zeta_i$.
\item (Quantum Leray--Hirsch) By \cite[Theorem~0.2]{lee2013invariance}, for a vector bundle $V \to B$ with projective base $B$, one has a prescribed relation between the atom of the fiberwise projectivization $\mathrm{Proj}_B(V)$ and the atom of the base $B$.
\end{enumerate}

Using these relations together with the Weak Factorisation Theorem, it follows that there exists a well-defined birationally invariant map
\begin{align*}
\mathrm{Atom}: \mathrm{ProjVar}_{\mathbb{C}} & \longrightarrow  \mathbb{N}^{\bigoplus_{\overline{c}} NSQC_{\overline{c}}}/\sim, \\
X & \longmapsto \mathrm{Atom}(X),
\end{align*}
where $\mathrm{ProjVar}_{\mathbb{C}}$ denotes the set of projective varieties over $\mathbb{C}$, and the target is the set of finite formal linear combinations of objects of the groupoids $NSQC_{\overline{c}}$ for all $\overline{c}$, modulo the equivalence above.

\begin{lemma}
The tensor product of connections induces a multiplication on $\mathbb{N}^{\bigoplus_{\overline{c}} NSQC_{\overline{c}}}/\sim$.
\end{lemma}
\begin{proof}
This follows from the fact that each of the atomic relations (1)--(3) is compatible with the tensor product of connections.
\end{proof}
By an abuse of notation, we will also call this operator tensor product. 
\begin{corollary}
\label{cor:atommult}
After a suitable formal change of quantum variables, one has
\[
\mathrm{Atom}(X \times Y) \;\sim\; \mathrm{Atom}(X) \otimes \mathrm{Atom}(Y).
\]
\end{corollary}
\begin{proof}
This is a reformulation of Theorem~\ref{thm:formtensoriso}.
\end{proof}

This multiplicativity property of atoms for products is stated and used in \cite[Section~5.2.6.1]{katzarkov2025birational}, although a detailed verification of the construction is omitted there. 

\subsection{Motivic measures}

Let $K_0(\mathrm{Var})$ denote the Grothendieck ring of quasi-projective varieties over $\mathbb{C}$. By definition, $K_0(\mathrm{Var})$ is the free abelian group generated by isomorphism classes $[X]$ of quasi-projective varieties, modulo the scissor relations
\[
[X] = [X \setminus Z] + [Z],
\]
for any Zariski closed subvariety $Z \subset X$. The multiplication is defined by the Cartesian product
\[
[X] \cdot [Y] := [X \times Y].
\]

A fundamental result due to Looijenga and Bittner asserts that $K_0(\mathrm{Var})$ admits a presentation in terms of smooth projective varieties only.

\begin{proposition}[{\cite{looijenga2000motivic, bittner2004universal}}]
\label{prop:bittnerpres}
The ring $K_0(\mathrm{Var})$ is isomorphic to the ring generated by isomorphism classes of smooth projective varieties, subject to the defining relations
\[
[Y] + [Z] = [X] + [E],
\]
where $Y$ is the blow-up of a smooth projective variety $X$ along a smooth center $Z \subset X$, and $E$ denotes the exceptional divisor.
\end{proposition}

This presentation allows one to pass from the Grothendieck ring to structures defined purely in terms of smooth projective geometry. In particular, Katzarkov--Kontsevich--Pantev--Yu observed in \cite[Section~5.3.2]{katzarkov2025birational} that there exists a natural map
\[
\begin{aligned}
\phi \colon K_0(\mathrm{Var}) & \longrightarrow \mathbb{Z}^{\bigoplus_{\overline{c}} NSQC_{\overline{c}}}/\sim \\
[X] & \longmapsto \mathrm{Atom}(X),
\end{aligned}
\]
where the target is the free abelian group generated by equivalence classes of objects in the groupoids $NSQC_{\overline{c}}$, for all $\overline{c}$, and the quotient is taken with respect to the relations of atoms discussed earlier. The coefficients lie in $\mathbb{Z}$ rather than $\mathbb{N}$, reflecting the fact that $K_0(\mathrm{Var})$ involves virtual classes and hence may contain negative contributions.

The definition of $\phi$ is made using the generators provided by Proposition~\ref{prop:bittnerpres}, namely smooth projective varieties. The map $\phi$ is claimed in \cite{katzarkov2025birational} to be a ring homomorphism. The additivity of $\phi$ is immediate: the blow-up relation for atoms (see Theorem~\ref{thm:quantumbu}) exactly mirrors the blow-up relation appearing in Proposition~\ref{prop:bittnerpres}, ensuring that $\phi$ is a well-defined group homomorphism. 
We now conclude that $\phi$ is indeed a ring homomorphism.

\begin{corollary}
\label{cor:phiring}
Equip the abelian group $\mathbb{Z}^{\bigoplus_{\overline{c}} NSQC_{\overline{c}}}/\sim$ with the multiplication induced by the tensor product of atoms. Then the map $\phi$
is a ring homomorphism.
\end{corollary}

\begin{proof}
The additivity of $\phi$ follows from the compatibility of the blow-up relation of atoms with Proposition~\ref{prop:bittnerpres}. Multiplicativity is guaranteed by Corollary~\ref{cor:atommult}, which identifies $\mathrm{Atom}(X \times Y)$ with $\mathrm{Atom}(X) \otimes \mathrm{Atom}(Y)$ after a suitable formal change of variables. Hence $\phi$ preserves both the addition and multiplication, proving that it is a ring homomorphism.
\end{proof}

\begin{remark}
The map $\phi$ may be interpreted as a new kind of \emph{motivic measure}, in the sense of Gillet--Soulé and Looijenga. Classical motivic measures, such as the Hodge--Deligne polynomial, send a variety $X$ to invariants of its cohomology. In contrast, $\phi$ records the structure of the quantum connection of $X$, via its decomposition into atoms. Thus $\phi$ enriches the motivic formalism with quantum invariants: it factors birational geometry through the Grothendieck ring while retaining an appropriate amount of GW-theoretic information encoded in the quantum connection.
\end{remark}

\begin{example}
Consider the projective line $\mathbb{P}^1$. Classically, the Hodge--Deligne polynomial assigns
\[
E(\mathbb{P}^1; u,v) = 1 + uv,
\]
reflecting that $H^0(\mathbb{P}^1) \cong \mathbb{C}$ and $H^2(\mathbb{P}^1) \cong \mathbb{C}(-1)$. In the quantum setting, the quantum connection of $\mathbb{P}^1$ splits into two rank-one summands, corresponding to eigenvalues $\pm \sqrt{q}$ of quantum multiplication by $c_1(T\mathbb{P}^1)$. Hence,
\[
\phi([\mathbb{P}^1]) = \mathrm{Atom}(\mathbb{P}^1) = [\mathcal{A}_+] + [\mathcal{A}_-],
\]
where $\mathcal{A}_\pm \in NSQC_{\overline{c}}$ are the two rank-one framed connections determined by the exponential factors $e^{\pm \sqrt{q}}$. 

\end{example}

\bibliographystyle{amsplain}
\bibliography{main}

\end{document}